\documentclass[a4paper,12pt]{article}
\usepackage{amssymb,amsmath,amsthm}
\usepackage{graphics,graphicx,color}
\usepackage{srcltx}
\DeclareGraphicsExtensions{.eps} \textwidth=169mm
\definecolor{darkred}{rgb}{0.9,0.1,0.1}

\newtheorem{proposition}{Proposition}[section]
\newtheorem{theorem}{Theorem}[section]
\newtheorem{lemma}{Lemma}[section]
\newtheorem{corollary}{Corollary}

{\rm}
\definecolor{darkred}{rgb}{0.9,0.1,0.1}

\def\eps{\varepsilon}

\definecolor{darkred}{rgb}{0.9,0.1,0.1}

\title{Periodic homogenization of non-local operators with a convolution type kernel}

\author{A. Piatnitski$^{\small\bf a,b}$ and E. Zhizhina$^{\small\bf a}$\\
\\
{%\footnotesize
\small $^{\small\bf a}$Institute for Information Transmission Problems of RAS,}\\[-1.5mm]
{\small 19, Bolshoi Karetnyi per. build.1,}\\[-1.5mm]
{\small 127051 Moscow, Russia}\\
{\small $^{\small\bf b}$University of Tromso, Campus in Narvik,}\\[-1.5mm]
{\small P.O.Box 385, Narvik 8505, Norway}\\
$ $}

\def\eps{\varepsilon}
\begin{document}

\maketitle

\begin{abstract}
The paper deals with homogenization problem for a non-local linear operator with
a kernel of convolution type in a medium with a periodic structure.
We consider the natural diffusive scaling of this operator and study the limit behaviour of the
rescaled operators as the scaling parameter tends to 0. More precisely we show that in the topology
of resolvent convergence the family of rescaled operators converges to a second order elliptic operator
with constant coefficients. We also prove the convergence of the corresponding semigroups both in $L^2$
space and the space of continuous functions,  and show that for the related family of Markov processes the invariance principle holds.
\end{abstract}

\setcounter{section}{-1}
\section{Introduction}

Recent time there is an increasing interest to the integral operators with a kernel of convolution type. These operators appear in many applications, such as models of population dynamics and the continuous contact model, where they describe the evolution of the density of a population, see for instance \cite{BCF, CCR, FKKMZ, KKP, KPZ} for the details. In these papers only the case of spacial homogeneous dispersal kernel has been investigated. We focus in this paper on the spacial inhomogeneous dispersal kernel depending both on the displacement $y-x$, and on the starting and the ending positions $x, y \in \mathbb R^d$.
We also mention here that the convolution type non-local operators describe the evolution of jump Markov processes (see for instance \cite{CCR}).
Although some of the properties of convolution type operators are similar to those of second order elliptic differential operators, there are also essential differences, for example, in the form of the fundamental solution for the corresponding nonlocal parabolic equation. In this connection it is interesting to understand which asymptotic properties of differential operators are inherited by nonlocal convolution type operators, and which are not. In this paper we study one of such question, namely we consider homogenization problem for convolution type operators in a periodic medium.
%========================================== \\

%The intensive study of these topics is motivated by the increasing interest to non-local models in %mathematical physics and stochastic analysis.

%=========================================== \\

 In this work we consider an integral convolution type operator of the form
\begin{equation}\label{L_u}
(L u)(x) \ = \ \lambda(x) \int\limits_{\mathbb R^d} a(x-y) \mu(y) (u(y) - u(x)) dy
\end{equation}
Here $\lambda(x)$ and $\mu(y)$ are bounded positive periodic functions characterizing the properties of the medium, and $a(z)$ is the jump kernel being a positive integrable function such that $a(-z)=a(z)$.  The detailed assumptions are given in  the next Section.

We then make a diffusive scaling of this operator
\begin{equation}\label{L_u_biseps}
(L^\eps u)(x) \ = \ \eps^{-d-2}\lambda\Big(\frac{x}{\eps}\Big) \int\limits_{\mathbb R^d} a\Big(\frac{x-y}{\eps}\Big) \mu\Big(\frac{y}{\eps}\Big) (u(y) - u(x)) dy,
\end{equation}
where $\eps$ is a positive scaling factor.
Our goal is to study the homogenization problem for operators $L^\eps$ that is to characterize the limit behaviour of $L^\eps$ as $\eps\to 0$.
\noindent

Homogenization theory of differential operators is a well-developed field, there is a vast literature
on this topic, we mention here the monographs \cite{BLP} and \cite{JKO}. In contrast with differential operators, the homogenization theory for convolution type integral operators and for more general integro-differential operators is not well-developed.

There are just several works in the existing mathematical literature devoted to homogenization problems for non-local operators describing the processes with jumps. Although in these works an essential progress has been achieved, there are still many interesting open problems in the area.

In \cite{Franke2013}  jump-diffusions
with periodic coefficients driven by stable L\'evy-processes with stability index
$\alpha>1$ were considered. It was shown that the limit process is an $\alpha$-stable L\'evy process
with an averaged jump-measure.

 The paper \cite{Rho_Var2008} deals with scaling limits of the solutions to stochastic differential
equations with stationary coefficients driven by Poisson random measures and Brownian
motions. The annealed convergence theorem is proved, in which the limit exhibits
a diffusive or superdiffusive behavior, depending on the integrability properties of
the Poisson random measure. It is important in this paper that the diffusion coefficient does not
degenerate.

In the recent work \cite{Sandric2016}  the homogenization problem for a Feller diffusion process with jumps generated by an integro-differential operator has been studied under the assumptions that the corresponding generator has rapidly periodically oscillating diffusion and jump coefficients, and under additional
regularity conditions.

It should be noted that in contrast with \eqref{L_u_biseps}, the generators considered in the quoted papers
have a non-zero diffusion part which improved the compactness properties of the corresponding resolvent.  Also, the kernels of the integral operators studied in these papers do not admit an oscillation in $y$ variable.

The goal of the present work is to prove homogenization result for the operators $L^\eps$.
More precisely, we are going to show that the family $L^\eps$ converges to a second order
divergence form elliptic operator with constant coefficient in the so-called $G$-topology that is
for any $m>0$ the family of operators $(-L^\eps+m)^{-1}$ converges strongly in $L^2(\mathbb R^d)$
to the operator $(-L^0+m)^{-1}$ where $L^0=\Theta^{ij}\frac{\partial^2}{\partial x^i\partial x^j}$
with a positive definite constant matrix $\Theta$. This is the subject of Theorem \ref{T1} in Section \ref{s_pbmset}.

As a consequence of this convergence we obtain the convergence of the corresponding semigroups.

Under additional regularity assumptions on the functions $a(x)$, $\lambda(x)$ and $\mu(x)$ the operator $L$
acts in the space of continuous functions in $\mathbb R^d$ that vanish at infinity. Also this operator generates a Markov process with trajectories in the space of c\'adl\'ag functions with values in $\mathbb R^d$, we   denote this space by  $D_{\mathbb R^d}[0, \infty)$. Our second aim is to show that under mentioned conditions the homogenization result
is also valid in the space of continuous functions in $\mathbb R^d$ and that under the diffusive scaling the invariance principle holds for the family of rescaled processes.

The methods used in the paper rely on asymptotic expansion techniques and constructing  periodic correctors of the first and second order.
Notice that, in contrast with the case of differential operators, in our case the coefficients of the auxiliary problem
on the periodicity cell differ from the coefficients of the original problem. This is an interesting feature of the studied non-local operators.

Another crucial feature of the non-local operators considered here is non-compactness of their resolvent. In this connection we cannot use the techniques based on the compactness of the family of solutions
and have to replace them with different arguments.

The paper is organized as follows. In Section \ref{s_pbmset} we provide the detailed setting of the problem
and formulate our homogenization result in the space $L^2(\mathbb R^d)$.

Then in Section \ref{s_corr} we introduce a number of auxiliary periodic problems, define correctors and prove some technical results.

Section \ref{s_prooft1} is devoted to the proof of the homogenization result in $L^2(\mathbb R^d)$.

The convergence of the corresponding semigroups is justified in Section \ref{s_semigroup}.
It implies the convergence of solutions to the corresponding evolution equations.

%{\tt Modest (still praising) words about homogenization in $C$ and invariance principle}

Finally, in Section \ref{s_markov} we study the operator $L$ in the space of continuous functions in
$\mathbb R^d$ that vanish at infinity. We show that, under natural regularity
assumptions on the coefficients of $L$, the homogenization result for operators $L^\eps$ remains valid
in this space of continuous functions, and that the corresponding  semigroups converge.
 Moreover, we prove that for the processes generated by $L^\eps$ the invariance principle
holds in the paths space.

\section{Problem setup}\label{s_pbmset}

In this section we provide all the conditions on the coefficients of operator $L$
and then formulate our homogenization result for the family $L^\eps$

For the function $a(z)$ we assume the following: $a(z) \in  L^{1}(\mathbb R^d) \cap L^{2}_{\rm loc}(\mathbb R^d)$, $a(z) \ge 0; \  a(-z) = a(z)$, and
\begin{equation}\label{M2}
\| a \|_{L^1(\mathbb R^d)}  = \int\limits_{\mathbb R^d} a(z) \, dz = a_1 >0; \qquad \int\limits_{\mathbb R^d} |z|^2 a(z) \, dz < \infty.
\end{equation}

Function $\lambda(x), \mu(x)$ are periodic and bounded from above and from below:
\begin{equation}\label{lm}
0< \alpha_1 \le \lambda(x), \  \mu(x) \le \alpha_2 < \infty.
\end{equation}
In what follows we identify periodic functions with functions defined on the torus $\mathbb{T}^d = [0,1]^d $. The operator $L$ is a bounded (non-symmetric) operator in $L^2(\mathbb R^d)$. Indeed, letting
$$
(L^-u)(x)=\lambda(x)\int_{\mathbb R^d}a(x-y)\mu(y)u(y)\,dy,
$$
we have for any $u\in C_0^\infty(\mathbb R^d)$
$$
\|(L^-u)\|^2_{L^2(\mathbb R^d)}=\int_{\mathbb R^d}dx\,\lambda^2 (x)\int_{\mathbb R^d}
a(x-y)\mu(y)u(y)\,dy\,\int_{\mathbb R^d}
a(x-z)\mu(z)u(z)\,dz
$$
$$
\leq \alpha_2^4\int_{\mathbb R^d}dx\,\int_{\mathbb R^d}
a(x-y)|u(y)|\,dy\,\int_{\mathbb R^d}
a(x-z)|u(z)|\,dz
$$
$$
=\alpha_2^4\int_{\mathbb R^d}
a(y)\,dy\,\int_{\mathbb R^d}
a(z)\,dz\int_{\mathbb R^d}|u(x+y)|\,|u(x+z)|dx\,\leq a_1^2\alpha_2^4\|u\|_{L^2(\mathbb R^2)}^2.
$$
Therefore, $L^-$ can be extended to a bounded operator acting from $L^2(\mathbb R^d)$ to $L^2(\mathbb R^d)$.
We denote the space of such operators by $\mathcal{L}(L^2(\mathbb R^d),L^2(\mathbb R^d))$. This implies the boundedness of $L$.

Let us consider the family of operators
\begin{equation}\label{L_eps}
(L^{\varepsilon} u)(x) \ = \ \frac{1}{\varepsilon^{d+2}} \int\limits_{\mathbb R^d} a \Big( \frac{x-y}{\varepsilon} \Big) \lambda \Big( \frac{x}{\varepsilon} \Big) \mu \Big( \frac{y}{\varepsilon} \Big) \big( u(y) - u(x) \big) dy.
\end{equation}
We are interested in the limit behavior of the operators $L^{\varepsilon}$ as $\varepsilon \to 0$. Since the norm of  $L^{\varepsilon}$ in $L^2(\mathbb R^d)$ tends to infinity, the limit operator if exists need not be bounded. We are going to show that the operators  $L^{\varepsilon}$ converge in the topology of  resolvent convergence. Let us fix an arbitrary $m>0$, %any $f \in L^2(\mathbb R^d)$,
and define $u^{\varepsilon}$ as the solution of equation:
\begin{equation}\label{u_eps}
(L^{\varepsilon} - m) u^{\varepsilon} \ = \ f, \quad \mbox{ i.e. } \; u^{\varepsilon} \ = \ (L^{\varepsilon} - m)^{-1} f,
\end{equation}
with $f \in L^2(\mathbb R^d)$. Denote by $L^0$ the following operator in  $L^2(\mathbb R^d)$:
\begin{equation}\label{L_hat}
L^0 u \ = \
%\sum_{i,j = 1}^d
\Theta^{i j} \frac{\partial^2 u}{\partial x^i   \partial x^j} \ = \ \Theta \nabla \nabla u, \quad {\cal D}( L^0) = H^2(\mathbb R^d)
\end{equation}
with a positive definite matrix $\Theta = \{ \Theta^{i j} \}, \ i,j = 1, \ldots, d,$ defined below, see (\ref{FA2}); here and in what follows we assume the summation over repeated indices. Let $u_0(x)$ be a solution of the equation
\begin{equation}\label{u_0}
%\sum_{i,j = 1}^d
\Theta^{i j} \frac{\partial^2 u_0}{\partial x^i   \partial x^j} - m u_0 = f,  \quad \mbox{ i.e. } \; u_0 \ = \ ( L^0 - m)^{-1} f
\end{equation}
with the same right-hand side $f$ as in (\ref{u_eps}).

Our first  result reads.

\begin{theorem}\label{T1} The family of resolvents  $(L^{\varepsilon} - m)^{-1}$ converges
strongly to the resolvent $(L^0 - m)^{-1}$, as $\eps\to0$, that is
for any $f \in L^2(\mathbb R^d)$ it holds:
\begin{equation}\label{t1}
\| (L^{\varepsilon} - m)^{-1} f - (L^0 - m)^{-1} f \|_{L^2(\mathbb R^d)} \ \to 0, \quad \mbox{  as } \; \varepsilon \to 0.
\end{equation}

\end{theorem}

\section{Correctors and auxiliary statements}
\label{s_corr}

In this section we introduce auxiliary periodic problems,  construct periodic correctors
and then make use of these correctors to approximate the solution $u^\eps$.

We consider first the case when $f \in {\cal S}(\mathbb R^d)$ and prove the convergence of the corresponding functions (\ref{u_eps}) to (\ref{u_0}):
\begin{equation}\label{convergence1}
\| u^{\varepsilon} - u_0 \|_{L^2(\mathbb R^d)} \ \to 0, \quad \mbox{ as } \ \varepsilon \to 0.
\end{equation}
\\

Let us consider a "small" perturbation $v^{\varepsilon}$ of the function $u_0$ defined by:
\begin{equation}\label{v_eps}
v^{\varepsilon}(x) \ = \ u_0(x)+ \varepsilon \varkappa_1 (\frac{x}{\varepsilon})\cdot \nabla u_0(x) + \varepsilon^2 \varkappa_2 (\frac{x}{\varepsilon})\cdot \nabla \nabla u_0(x)
\end{equation}
with  periodic vector function $\varkappa_1(x) \in  (L^2({\mathbb T^d}))^d$, i.e. $\varkappa_1^{i}(x) \in L^2(\mathbb T^d), \; \forall i=1, \ldots, d,$ and periodic matrix function $\varkappa_2(x) \in (L^2({\mathbb T^d}))^{d^2}$, i.e. $\varkappa_2^{i j}(x) \in L^2(\mathbb T^d), \; \forall i,j = 1, \ldots, d,$ that will be defined below, see eq. (\ref{kappa_1}) and (\ref{kappa_2}).
In (\ref{v_eps}) an later on $\varkappa_1\cdot \nabla u_0$ stands for $\varkappa_1^{i}\frac{\partial}{\partial x^i}u_0$ and  $\varkappa_2\cdot \nabla \nabla u_0$ stands for $\varkappa_2^{ij}
\frac{\partial^2}{\partial x^i\partial x^j}u_0$. In particular, in the case $d=1$
\begin{equation}\label{v_eps_1}
v^{\varepsilon}(x) \ = \ u_0(x)+ \varepsilon \varkappa_1 (\frac{x}{\varepsilon}) u'_0(x) + \varepsilon^2 \varkappa_2 (\frac{x}{\varepsilon}) u''_0(x)
\end{equation}
Since $L^0$ is an elliptic operator with constant coefficients and $f \in {\cal S}(\mathbb R^d) $, then  $u_0 \in {\cal S} (\mathbb R^d) $ and $v^{\varepsilon}$ is correctly defined.
After substitution (\ref{v_eps}) to (\ref{L_eps}) we get
$$
\begin{array}{rl}
\displaystyle
(L^{\varepsilon} v^{\varepsilon})(x) \!&\!
\displaystyle
= \ \frac{1}{\varepsilon^{d+2}} \int\limits_{\mathbb R^d} a \Big( \frac{x-y}{\varepsilon} \Big) \lambda \Big( \frac{x}{\varepsilon} \Big) \mu \Big( \frac{y}{\varepsilon} \Big)
\bigg\{ u_0(y)+ \varepsilon \varkappa_1 \Big(\frac{y}{\varepsilon}\Big)\cdot \nabla u_0(y) +
\\[3mm] \!&\! \displaystyle +\
\varepsilon^2 \varkappa_2 \Big(\frac{y}{\varepsilon}\Big)\cdot \nabla \nabla u_0(y) -
u_0(x)-\varepsilon \varkappa_1 \Big(\frac{x}{\varepsilon}\Big)\cdot \nabla u_0(x) - \varepsilon^2 \varkappa_2 \Big(\frac{x}{\varepsilon}\Big)\cdot \nabla \nabla u_0(x) \bigg\} dy.
\end{array}
$$

\begin{lemma}\label{ml} (The main lemma)  Assume that $f \in {\cal{S}}(\mathbb R^d)$, then there exist functions $\varkappa_1 \in (L^2(\mathbb T^d))^d$  and $\varkappa_2 \in (L^2(\mathbb T^d))^{d^2}$ (a vector function $\varkappa_1$ and a matrix function $\varkappa_2$) and a positive definite matrix $\Theta$ such that that  for the function $v^{\varepsilon}$ defined by (\ref{v_eps}) we have
\begin{equation}\label{mle}
L^{\varepsilon} v^{\varepsilon} \ = \
%\sum_{i,j = 1}^d
\Theta^{ij} \frac{\partial^2 u_0}{\partial x^i  \partial x^j} \ + \ \phi_\varepsilon, \quad \mbox{ where } \; \lim\limits_{\varepsilon \to 0}\| \phi_\varepsilon \|_{L^2(\mathbb R^d)}= 0.
\end{equation}

\end{lemma}

\begin{proof}

After change of variables $\frac{x-y}{\varepsilon}=z$ we get
\begin{equation}\label{ml_1}
\begin{array}{r}\displaystyle
\!\!\!(L^{\varepsilon} v^{\varepsilon})(x) \ = \ \frac{1}{\varepsilon^{2}} \int\limits_{\mathbb R} dz \  a (z) \lambda \Big( \frac{x}{\varepsilon} \Big) \mu \Big( \frac{x}{\varepsilon} -z \Big) \bigg\{ u_0(x-\varepsilon z)+ \varepsilon \varkappa_1 \Big(\frac{x}{\varepsilon}-z \Big)\cdot \nabla u_0 (x-\varepsilon z)
\\[4mm] \displaystyle
+\,\varepsilon^2 \varkappa_2 \Big( \frac{x}{\varepsilon}-z \Big)\cdot \nabla \nabla u_0(x-\varepsilon z) - u_0(x)-\varepsilon \varkappa_1 \Big( \frac{x}{\varepsilon} \Big)\cdot\nabla u_0(x) - \varepsilon^2 \varkappa_2 \Big(\frac{x}{\varepsilon} \Big)\cdot \nabla \nabla u_0(x) \bigg\}.
\end{array}
\end{equation}
Using the following identity based on the integral form of the remainder term in the Taylor expansion
$$
u(y) \ = \ u(x) + \int_0^1 \frac{\partial}{\partial t} \ u(x+(y-x)t) \ dt \ = \ u(x) + \int_0^1 \nabla u (x+  (y-x)t) \cdot (y-x) \ dt,
$$
$$
u(y) \ = \ u(x) + \nabla u(x) \cdot (y-x) + \int_0^1  \nabla \nabla u(x+(y-x)t) (y-x)\cdot (y-x) (1-t) \ dt
$$
which is valid for any $x, y \in \mathbb R^d$, we can rearrange (\ref{ml_1}) as follows
\begin{equation*}%\label{K1_1}
\begin{array}{rl}
(\!\!\!&\!\!L^{\varepsilon} v^{\varepsilon})(x) =
\\[2mm]
&\!\!= \displaystyle
  \frac{1}{\varepsilon^{2}}\! \int\limits_{\mathbb R^d}\! dz \, a(z) \lambda \Big( \frac{x}{\varepsilon} \Big) \mu \Big( \frac{x}{\varepsilon}\! -\!z \Big)\! \bigg\{( u_0(x) - \varepsilon z\cdot\nabla u_0(x) + \varepsilon^2\! \int\limits_0^{1}\!  \nabla \nabla u_0(x-\varepsilon z t)\cdot z\!\otimes\!z\, (1\!-\!t) \, dt
\\[5mm]
&\!\!+\, \displaystyle\varepsilon \varkappa_1 \Big(\frac{x}{\varepsilon}-z \Big)\cdot \Big( \nabla u_0(x)\!-\!\varepsilon  \nabla \nabla u_0 (x)\,z + \varepsilon^2 \int_0^{1}  \nabla \nabla \nabla u_0(x-\varepsilon z t) z\!\otimes\!z (1-t) \ dt \Big) \\[4mm]
&\!\!+\,\displaystyle
\varepsilon^2 \varkappa_2 \Big( \frac{x}{\varepsilon}-z \Big)\cdot \nabla \nabla u_0(x-\varepsilon z) \ - \
u_0(x)-\varepsilon \varkappa_1 \Big( \frac{x}{\varepsilon} \Big)\cdot \nabla u_0(x) - \varepsilon^2 \varkappa_2 \Big(\frac{x}{\varepsilon} \Big)\cdot \nabla \nabla u_0(x) \bigg\},
\end{array}
\end{equation*}
where $z\otimes z=\{z^i z^j\}\big|_{i,j=1}^d$,\ $\ \nabla \nabla u_0(\cdot)z=\frac{\partial^2 u_0}{\partial x^i\partial x^j}(\cdot) z^j$, and $\nabla \nabla \nabla u_0(\cdot) z\!\otimes\!z =\frac{\partial^3 u_0}{\partial x^i\partial x^j\partial x^k}(\cdot) z^jz^k$. Collecting power-like terms in the last relation
we obtain
\begin{eqnarray}
(L^{\varepsilon} v^{\varepsilon})(x) \hskip -1.7cm &&\nonumber\\[1.6mm]
\label{K1_1}
&&\!\!\!\!\!=\, \frac{1}{\varepsilon} \lambda \Big( \frac{x}{\varepsilon} \Big)\nabla u_0(x)\! \cdot\! \int\limits_{\mathbb R^d}  \Big[ -z + \varkappa_1 \Big(\frac{x}{\varepsilon}-z \Big) - \varkappa_1 \Big(\frac{x}{\varepsilon}\Big) \Big]  a (z) \mu \Big( \frac{x}{\varepsilon} -z \Big) \, dz
\\[1mm]
\label{K2_1}
&&\!\!\!\!\! +\,\lambda \Big(\! \frac{x}{\varepsilon}\! \Big)\nabla \nabla u_0 (x)\!\cdot\!  \int\limits_{\mathbb R^d}\! \Big[ \frac12 z\!\otimes\!z\! - z \!\otimes\!\varkappa_1 \Big(\frac{x}{\varepsilon}\!-\!z \Big) + \varkappa_2 \Big( \frac{x}{\varepsilon}\! -\! z \Big) \!- \varkappa_2 \Big(\frac{x}{\varepsilon}\Big) \Big]   a (z) \mu \Big( \frac{x}{\varepsilon}\! -\!z \Big) \, dz
\\[0.5mm] \nonumber &&\!\!\!\!\!
 +\, \ \phi_\varepsilon (x) \hfill
\end{eqnarray}
with
\begin{equation}\label{14}
\begin{array}{rl} \displaystyle
\!\!\!\!&\hbox{ }\!\!\!\!\!\!\!\!\!\!\!\!\phi_\varepsilon (x) =\\[3mm]
& \!\!\!\!\!\!\!\!\displaystyle
 \varepsilon^{-2}\!\! \int\limits_{\mathbb R^d}\! dz \, a (z) \lambda \Big( \frac{x}{\varepsilon} \Big) \mu \Big( \frac{x}{\varepsilon}\! -\!z \Big)  \bigg\{ \varepsilon^2\! \int\limits_0^{1}  \nabla \nabla u_0(x-\varepsilon z t) \!\cdot\! z\!\otimes\!z \,(1-t) \ dt  - \frac{\varepsilon^2}2 \nabla \nabla u_0(x)\!\cdot\!  z\!\otimes\!z \
\\[4mm]  &\!\!\!\!\!\!\!\!\! \displaystyle
+\, \varepsilon^3 \varkappa_1 \Big(\frac{x}{\varepsilon}\!-\!z \Big)\!\cdot\! \int\limits_0^{1}\!  \nabla \nabla \nabla u_0(x\!-\!\varepsilon z t) z\!\otimes\!z (1\!-\!t) \, dt  \, - \, \varepsilon^3 \varkappa_2 \Big(\frac{x}{\varepsilon}\!-\!z \Big) \!\cdot\! \int\limits_0^{1}\!  \nabla \nabla \nabla u_0(x\!-\!\varepsilon z t)z  \, dt\!  \bigg\}
\end{array}
\end{equation}
Next we prove that  $\|\phi_\varepsilon\|\big._{L^2(\mathbb R^d)}$ is vanishing as $\varepsilon \to 0$.

\begin{lemma}\label{fi_0} Let $u_0 \in {\cal{S}}(\mathbb R^d)$, and assume that  $\lambda, \mu$ are periodic functions satisfying bounds (\ref{lm}), and all the components of $\varkappa_1$ and $\varkappa_2$ are elements of $ L^2(\mathbb T^d)$. Then
\begin{equation}\label{fi}
\| \phi_\varepsilon (x) \|\big._{L^2(\mathbb R^d)^d} \ \to \ 0, \quad \mbox{ as } \; \varepsilon \to 0.
\end{equation}
\end{lemma}

\begin{proof}
The first term on the right-hand side in (\ref{14}) (the term of order $\eps^0$) reads
\begin{equation}\label{dec1}
\begin{array}{rl}
\displaystyle
\phi_\varepsilon^{(1)}&\!\!\!\! \displaystyle
 (x)=
\\[3mm]
  =\!\!\!&\displaystyle \frac{1}{\varepsilon^{2}} \int\limits_{\mathbb R^d} \!dz \, a (z) \lambda \Big(\! \frac{x}{\varepsilon} \Big) \mu \Big( \frac{x}{\varepsilon} -z \Big) \varepsilon^2 \int\limits_0^{1}  \Big( \nabla \nabla u_0(x - \varepsilon z t) - \nabla \nabla u_0(x) \Big)\cdot z\!\otimes\!z (1-t) \, dt
\\[3mm]
=\!\!\!&\displaystyle
\int\limits_{|z| \le R} dz \ a (z) \lambda \Big( \frac{x}{\varepsilon} \Big) \mu \Big( \frac{x}{\varepsilon} -z \Big)
\int_0^{1} \ \Big( \nabla \nabla u_0(x - \varepsilon z t) - \nabla \nabla u_0(x) \Big)\cdot z\!\otimes\!z(1-t) \, dt
\\[3mm]
+\!\!\!&\displaystyle
 \int\limits_{|z| > R} dz \ a (z) \lambda \Big( \frac{x}{\varepsilon} \Big) \mu \Big( \frac{x}{\varepsilon} -z \Big)
\int_0^{1} \ \Big( \nabla \nabla u_0(x - \varepsilon z t) - \nabla \nabla u_0(x) \Big)\cdot z\!\otimes\!z (1-t) \ dt.
\\[9mm]
:=\!\!\!&\displaystyle \phi_\varepsilon^{(1, \le R)} (x)\,+\,\phi_\varepsilon^{(1, > R)} (x).
\end{array}
\end{equation}
Then
$$
\| \phi_\varepsilon^{(1, \le R)}  \|\big._{L^2(\mathbb R^d)} \ \le \ \alpha_2^2 \  \sup\limits_{|z| \le R} \| \nabla \nabla u_0(x - \varepsilon z t) - \nabla \nabla u_0(x) \|\big._{L^2(\mathbb R^d)^{d^2}} \int_{\mathbb R^d} |z|^2 \  a(z) \int_0^{1} (1-t) \, dt \, dz
$$
$$
=\,\frac{\alpha_2^2}{2} \sup\limits_{|z| \le  R} \| \nabla \nabla u_0(x - \varepsilon z t) - \nabla \nabla u_0(x) \|\big._{L^2(\mathbb R^d)^{d^2}} \int_{\mathbb R^d} |z|^2 \ a(z) \ dz
$$
and
$$
\| \phi_\varepsilon^{(1, > R)}  \|\big._{L^2(\mathbb R^d)} \ \le \ 2 \alpha_2^2 \|  \nabla \nabla u_0 (x) \|\big._{L^2(\mathbb R^d)^{d^2}} \int_{|z|>R} |z|^2 \  a(z) \ dz.
$$
If we take $R=R(\varepsilon) = \frac{1}{\sqrt{\varepsilon}}$, then both
$$
\| \phi_\varepsilon^{(1, \le R(\varepsilon))} \|\big._{L^2(\mathbb R^d)} \to 0 \quad  \mbox{and} \quad
\| \phi_\varepsilon^{(1, > R(\varepsilon))} \|\big._{L^2(\mathbb R^d)} \to 0, \quad \mbox{ as } \; \varepsilon \to 0.
$$
This yields
\begin{equation}\label{phi_1bis}
\| \phi_\varepsilon^{(1)} \|\big._{L^2(\mathbb R^d)} \to 0, \quad \mbox{ as } \; \varepsilon \to 0.
\end{equation}
For the second term on the right-hand side of \eqref{14}
$$
\phi_\varepsilon^{(2)}(x) = \varepsilon \int\limits_{\mathbb R^d} dz \, a (z) \lambda \Big( \frac{x}{\varepsilon} \Big) \mu \Big( \frac{x}{\varepsilon} -z \Big)   \varkappa_1 \Big(\frac{x}{\varepsilon}-z \Big)\!\cdot\!
\int_0^{1}  \nabla \nabla \nabla u_0(x-\varepsilon z t) z\!\otimes\!z\, (1-t) \, dt
$$
we have
\begin{equation}\label{phi_2}
\| \phi_\varepsilon^{(2)} (x) \|\big._{L^2(\mathbb R^d)} \, \le \, \frac{\varepsilon}{2}  \alpha_2^2 \, \sup\limits_{z, q \in \mathbb R^d}  \big\|  \varkappa_1 \Big(\frac{x}{\varepsilon}-z \Big) \  \nabla \nabla \nabla u_0 (x-\varepsilon z+q) \big\|_{L^2(\mathbb R^d)^{d^2}} \int\limits_{\mathbb R^d} |z|^2 \, a(z) \, dz.
\end{equation}
We estimate now $\sup\limits_{z, q \in \mathbb R^d}  \|  \varkappa_1 \big(\frac{x}{\varepsilon}-z \big) \,  \nabla \nabla \nabla u_0 (x-\varepsilon z+q) \|\big._{L^2(\mathbb R^d)^{d^2}}$. Taking $y=x-\varepsilon z$ and considering the fact that the function $\varkappa_1$ is periodic we get
$$
\sup\limits_{q \in \mathbb R^d}  \|  \varkappa_1 \Big(\frac{y}{\varepsilon} \Big) \ \nabla \nabla \nabla u_0 (y+q) \|\big._{L^2(\mathbb R^d)^{d^2}} = \sup\limits_{q \in \varepsilon \mathbb T^d}  \|  \varkappa_1 \Big(\frac{y}{\varepsilon} \Big) \ \nabla \nabla \nabla u_0 (y + q) \|\big._{L^2(\mathbb R^d)^{d^2}}.
$$
Let us show that this quantity admits a uniform in $\eps$ upper bound. Indeed, denoting $I_k (\varepsilon) = \varepsilon k + \varepsilon \mathbb T^d, \; k \in \mathbb Z^d$ with $\mathbb{T^d} = [0,1]^d$,  we have
$$
\sup\limits_{q \in \varepsilon \mathbb T^d}  \|  \varkappa_1 \Big(\frac{y}{\varepsilon} \Big) \ \nabla \nabla \nabla u_0 (y+q) \|\big.^2_{(L^2(\mathbb R^d))^{d^2}}
$$
$$
 \leq\,
\sup\limits_{q \in \varepsilon \mathbb T^d} \sum_{i,j,l,m=1}^d\sum_{k \in \mathbb Z^d} \int_{I_k(\varepsilon)}  \big[\varkappa^{i}_1 \Big(\frac{y}{\varepsilon} \Big)\big]^2 \,  \big[\partial_{x^j} \partial_{x^l} \partial_{x^m} u_0 (y+q)\big]^2 \ dy
$$
$$
\leq \,
\sum_{j,l,m=1}^d
\sum_{k \in \mathbb Z^d} \ \max_{y \in I_k (\varepsilon), \ q \in \varepsilon \mathbb T^d}  \big[\partial_{x^j} \partial_{x^l} \partial_{x^m} u_0 (y+q)\big]^2 \int_{I_k(\varepsilon)}  \varkappa^2_1 \Big(\frac{y}{\varepsilon} \Big) \ dy   \  =
$$
$$
\| \varkappa_1 \|\big.^2_{(L^2(\mathbb T^d))^d} \ \varepsilon^d\, \sum_{j,l,m=1}^d \sum_{k \in \mathbb Z^d} \max_{y \in I_k (\varepsilon), \ q \in \varepsilon \mathbb T^d} \big[\partial_{x^j} \partial_{x^l} \partial_{x^m} u_0 (y+q)\big]^2 \, \longrightarrow
$$
$$
\longrightarrow    \| \varkappa_1 \|\big.^2_{(L^2(\mathbb T^d))^d} \sum_{j,l,m=1}^d \| \partial_{x^j} \partial_{x^l} \partial_{x^m} u_0 \|^2_{L^2(\mathbb R^d)},
$$
as $\varepsilon \to 0$. Here we have used the fact that for a functions $\psi \in {\cal{S}}(\mathbb R^d)$
$$
\varepsilon^d \ \sum_{k \in \mathbb Z^d} \max_{y \in I_k (\varepsilon), \ q \in \varepsilon \mathbb T^d}  \psi (y+q) \ \to \   \int\limits_{\mathbb R^d} \psi(x) \ dx, \quad \varepsilon \to 0.
$$
Thus from estimate (\ref{phi_2}) it follows that $\| \phi_\varepsilon^{(2)} \|\big._{L^2(\mathbb R^d)} \to 0$, as $ \varepsilon \to 0$.

Similarly for the third term on the right-hand side of (\ref{14}) we have
$$
\| \phi_\varepsilon^{(3)}(x) \|_{L^2(\mathbb R^d)} \,= \hskip 10cm
$$
\begin{equation}\label{phi_3}
= \, \varepsilon \, \Big\|  \int\limits_{\mathbb R^d} dz \ a (z) \lambda \Big( \frac{x}{\varepsilon} \Big) \mu \Big( \frac{x}{\varepsilon} -z \Big)  \varkappa_2
\Big(\frac{x}{\varepsilon}-z \Big)\!\cdot\! \int_0^1 \nabla \nabla \nabla u_0(x-\varepsilon z t) \ z  \ dt \Big\|_{L^2(\mathbb R^d)}
\end{equation}
$$
\leq \varepsilon c(d) \alpha_2^2 \left( \int\limits_{\mathbb R^d} |z| \ a (z) \ dz  \right) \ \sup\limits_{q \in \mathbb R^d}  \|  \varkappa_2 \Big(\frac{y}{\varepsilon} \Big) \ \nabla \nabla \nabla u_0 (y+q) \|\big._{(L^2(\mathbb R^d))^d} \ \le \ \varepsilon C_3
$$
for all sufficiently small $\varepsilon$, since as above we have for all $i,j,l,m,n$
$$
\| \varkappa^{i n}_2 \Big(\frac{y}{\varepsilon} \Big) \  \partial_{x^j} \partial_{x^l} \partial_{x^m} u_0 (y+q) \|^2_{L^2(\mathbb R^d)} \ \to \    \| \varkappa^{i n}_2 \|^2_{L^2(\mathbb T^d)} \| \partial_{x^j} \partial_{x^l} \partial_{x^m} u_0 \|^2_{L^2(\mathbb R^d)},
$$
as $\varepsilon \to 0$, uniformly in $q\in\mathbb R^d$.

\end{proof}

Our next step of the proof deals with constructing the correctors $\varkappa_1$ and $\varkappa_2$.

Denote $\xi=\frac{x}{\varepsilon}$ a variable on the period: $\xi \in \mathbb{T}^d = [0,1]^d$, then $\lambda(\xi), \mu(\xi), \varkappa_1(\xi), \varkappa_2(\xi)$ are functions on $\mathbb T^d$ and (\ref{K1_1}) - (\ref{K2_1}) can be understood as equations for the functions  $\varkappa_1(\xi), \varkappa_2(\xi), \ \xi \in \mathbb T^d$ on the torus.

We collect all the terms of the order $\varepsilon^{-1}$ in (\ref{K1_1}) and equate them to 0. This yields the following equation for the vector function $\varkappa_1(\xi) = \{ \varkappa_1^i (\xi) \}, \ \xi \in \mathbb T^d, \; i=1, \ldots, d,$ as unknown function:
%\comment{$\varkappa^i$ or $\varkappa^{(i)}$}
\begin{equation}\label{K1_2}
\int\limits_{\mathbb R^d}  \Big(  -z^i + \varkappa^i_1 (\xi-z) - \varkappa^i_1 (\xi) \Big) \ a (z) \mu (\xi -z ) \ dz = 0 \quad \forall \ i = 1, \ldots, d.
\end{equation}
Here $\varkappa_1(q), \ q \in \mathbb R^d$, is the periodic extension of  $\varkappa_1(\xi), \ \xi \in \mathbb T^d$. Notice that (\ref{K1_2}) is a system of uncoupled equations.
After change of variables  $q=\xi-z \in \mathbb R^d$ equation (\ref{K1_2}) can be written in the vector form as follows
\begin{equation}\label{kappa_1bis}
\int\limits_{\mathbb R^d}  a (\xi-q) \mu (q)  (\varkappa_1 (q) - \varkappa_1 (\xi) ) \ dq \ = \ \int\limits_{\mathbb R^d}  a (\xi-q) (\xi - q ) \mu (q ) \ dq,
\end{equation}
or
\begin{equation}\label{kappa_1}
A \varkappa_1 \ = \ f
\end{equation}
with the operator $A$ in $(L^2(\mathbb T^d))^d$ defined by
\begin{equation}\label{Akappa}
(A \bar\varphi) (\xi) \ = \ \int\limits_{\mathbb R^d}  a (\xi-q) \mu (q)  (\bar\varphi (q) - \bar\varphi (\xi) ) \ dq \ = \ \int\limits_{\mathbb T^d}  \hat a (\xi-\eta) \mu (\eta)  (\bar\varphi (\eta) - \bar\varphi (\xi) ) \ d\eta,
\end{equation}
and
%$\bar\varphi \in (L^2(\mathbb T^d))^d $ with
\begin{equation}\label{hata}
\hat a(\eta) \ = \ \sum_{k \in \mathbb Z^d} a(\eta +k), \quad \eta \in \mathbb T^d.
\end{equation}
Observe that the vector function
\begin{equation}\label{f}
f(\xi) \ = \ \int\limits_{\mathbb R^d}  a (\xi-q) \mu (q )  (\xi - q ) \ dq \ \in \ (L^2(\mathbb T^d))^d,
\end{equation}
because the function $f(\xi)$ is bounded for all $\xi \in \mathbb T^d$:
$$
\left| \int\limits_{\mathbb R^d}  a (\xi-q) (\xi - q ) \mu (q) \ dq \right| \ \le \ \alpha_2 \int\limits_{\mathbb R^d}  a (z) |z| \ dz \ < \ \infty.
$$
In \eqref{kappa_1} the operator $A$ applies component-wise.  In what follows, abusing slightly the notation, we use the same notation $A$ for the scalar operator in $L^2(\mathbb T^d)$ acting on each component in \eqref{kappa_1}.

Let us denote
$$
K \varphi (\xi)=   \int\limits_{\mathbb R^d}  a (\xi-q) \mu (q) \varphi (q)  \, dq,\qquad
\varphi \in L^2(\mathbb T^d).
$$
%the operator $A$ acting on each component $\varphi^i \in L^2(\mathbb T^d), \ i = 1, \ldots, d$.

\begin{lemma}\label{compact}
The operator
\begin{equation}\label{Kkappa}
K \varphi (\xi) \ = \  \int\limits_{\mathbb R^d}  a (\xi-q) \mu (q) \varphi (q)  \ dq \ = \   \int\limits_{\mathbb T^d}  \hat a (\xi-\eta) \mu (\eta) \varphi(\eta) \ d\eta, \quad  \varphi \in L^2(\mathbb T^d),
\end{equation}
is a compact operator in $L^2 (\mathbb T^d)$.
\end{lemma}

\begin{proof}
First we prove that $K$ is the bounded operator in $L^2 (\mathbb T^d)$. The set of bounded functions $\mathbf B (\mathbb T^d) \subset L^2 (\mathbb T^d)$ is dense in $L^2 (\mathbb T^d)$. Let $\varphi \in \mathbf B (\mathbb T^d)$, then the integral
$$
\bigg|\int\limits_{\mathbb R^d}  a (\xi-q) \mu (q) \varphi (q) \ dq\bigg| \ \le \  \alpha_2 \ a_1 \ \max |\varphi (q)|
$$
is bounded. Using Fubini's theorem and denoting $ w(q) = \mu(q) \varphi (q)$ we get
\begin{equation}\label{Kw}
\begin{array}{c}
\displaystyle
\|K \varphi \|^2_{L^2 (\mathbb T^d)} \ = \ \int\limits_{\mathbb T^d} \int\limits_{\mathbb R^d}  a (q-\xi)  w (q)  \ dq \  \int\limits_{\mathbb R^d}  a (q'-\xi)  w (q')  \ dq' \ d\xi
\\[4mm] \displaystyle
=\,\int\limits_{\mathbb R^d} \int\limits_{\mathbb R^d}  a (z) a(z') \left( \int\limits_{\mathbb T^d}  w (\xi+z)  w(\xi+z') d\xi \right) dz \ dz'
 \\[4mm] \displaystyle
 \leq \|  w \|^2_{L^2 (\mathbb T^d)} \left( \int\limits_{\mathbb R^d}  a (z) dz \right)^2 \le \alpha_2^2 \|a\|^2_{L^1(\mathbb R^d)}  \| \varphi \|^2_{L^2 (\mathbb T^d)}.
\end{array}
\end{equation}
 Consequently the operator $K$ can be expanded on $L^2 (\mathbb T^d) $ and we have:
$$
\|K \varphi \|_{L^2 (\mathbb T^d)} \ \le \ \alpha_2 \| a\|_{L^1 (\mathbb R^d)}  \| \varphi \|_{L^2 (\mathbb T^d)}, \quad \varphi \in L^2 (\mathbb T^d),
$$
or
\begin{equation}\label{K_norm}
\|K  \|_{\mathcal{L}(L^2 (\mathbb T^d),L^2 (\mathbb T^d))} \ \le \ \alpha_2  \| a \|_{L^1 (\mathbb R^d)}.
\end{equation}
\\
To prove the compactness of $K$ we consider approximations of $K$ by the following compact operators:
$$
(K_N \varphi) (\xi) \ = \  \int\limits_{\mathbb R^d}  a_N (\xi-q) \mu (q) \varphi (q)  \ dq \quad \mbox{ with } \;
a_N(z) = a(z) \cdot \chi_{[-N,N]^d}(z).
$$
Since $a-a_N \in L^1 (\mathbb R^d)$, then using (\ref{K_norm}) we get
$$
\|K-K_N \|_{L^2 (\mathbb T^d)}  \le   \alpha_2 \| a-a_N \|_{L^1 (\mathbb R^d)}.
$$
Consequently, $\| K-K_N \|_{\mathcal{L}(L^2 (\mathbb T^d),L^2 (\mathbb T^d))} \to 0$, as $N \to \infty$, and $K$ is a compact operator as the limit of the compact operators $K_N$.
\end{proof}

The operator
\begin{equation}\label{Gkappa}
G \varphi (\xi) \ = \  \varphi (\xi) \ \int\limits_{\mathbb R^d}  a (\xi-q) \mu (q)  \ dq \ = \ \varphi (\xi)  \int\limits_{\mathbb T^d}  \hat a (\xi-\eta) \mu (\eta)  \ d\eta, \quad \varphi \in L^2 (\mathbb T^d),
\end{equation}
is the operator of multiplication by the function $G(\xi) = \int\limits_{\mathbb R^d}  a (\xi-q) \mu (q)  \ dq$. Observe that
$$
0 < g_0 \le G(\xi) \le g_2 <\infty.
$$
Thus, the operator $A$ in (\ref{Akappa}) is the sum $A=G+K$, where $G$ and $K$ were defined in (\ref{Gkappa}) and (\ref{Kkappa}). Therefore $A$ is the sum of a positive reversible operator $G$ and a compact operator $K$, and the Fredholm theorem applies to (\ref{kappa_1}). It is easy to see that $Ker \ A^\ast = \{ \mu (\xi) \}$, then the solvability condition for (\ref{kappa_1}) takes the form:
\begin{equation}\label{FA1}
\int\limits_{\mathbb T^d} f(\xi) \mu(\xi) \ d \xi \  = \ 0.
\end{equation}
The validity of condition (\ref{FA1}) for the function $f$ defined in (\ref{f}) immediately follows from Lemma (\ref{l_1}).

\begin{lemma}\label{l_1}
For any periodic functions $\mu(y), \ \lambda(y), \; y \in \mathbb R^d$ we have: \\
if $a(x-y)=a(y-x)$, then
\begin{equation}\label{L1_1}
\int\limits_{\mathbb R^d} \int\limits_{\mathbb T^d}  a (x-y) \mu (y ) \lambda (x) \ dx \ dy = \int\limits_{\mathbb R^d} \int\limits_{\mathbb T^d}  a (x-y) \mu (x) \lambda (y) \ dx \ dy;
\end{equation}
if $b(x-y)= - b (y-x)$, then
\begin{equation}\label{L1_2}
\int\limits_{\mathbb R^d} \int\limits_{\mathbb T^d}  b (x-y) \mu (y ) \lambda (x) \ dx \ dy = - \int\limits_{\mathbb R^d} \int\limits_{\mathbb T^d}  b (x-y) \mu (x) \lambda (y) \ dx \ dy.
\end{equation}
\end{lemma}

\begin{proof}
Using periodicity of $\mu$ and $\lambda$ we get for any $z \in \mathbb R^d$:
$$
\int\limits_{\mathbb T^d}  \mu (z+x) \lambda (x) \ dx = \int\limits_{\mathbb T^d}  \mu (u) \lambda (u-z) \ du.
$$
Consequently, using the relation $a(x-y)=a(y-x)$ we have
$$
\int\limits_{\mathbb R^d} \int\limits_{\mathbb T^d}  a (y-x) \mu (y ) \lambda (x) \ dx \ dy = \int\limits_{\mathbb R^d} \int\limits_{\mathbb T^d}  a (z) \mu (z+x) \lambda (x) \ dx \ dz = \int\limits_{\mathbb R^d} \int\limits_{\mathbb T^d}  a (x-y) \mu (x) \lambda (y) \ dx \ dy.
$$
Similarly using that $b(x-y)= - b (y-x)$ we get
$$
\int\limits_{\mathbb R^d} \int\limits_{\mathbb T^d}  b (x-y) \mu (y ) \lambda (x) \ dx \ dy = - \int\limits_{\mathbb R^d} \int\limits_{\mathbb T^d}  b (y-x) \mu (y ) \lambda (x) \ dx \ dy =
$$
$$
- \int\limits_{\mathbb R^d} \int\limits_{\mathbb T^d}  b (z) \mu (z+x) \lambda (x) \ dx \ dz = - \int\limits_{\mathbb R^d} \int\limits_{\mathbb T^d}  b (x-y) \mu (x) \lambda (y) \ dx \ dy.
$$
\end{proof}

Thus, the solution $\varkappa_1(\xi)$ of equation (\ref{kappa_1}) exists  and is unique up to a constant vector. In order to fix the choice of this vector we assume that the average of each component of $\varkappa_1(\xi)$ over the period is equal to 0. \\

At the next step we collect in (\ref{K2_1}) the terms of the order $\varepsilon^0$. Our goal is to find the matrix function $\varkappa_2(\xi) = \{ \varkappa_2^{ij}(\xi) \}, \; \varkappa_2^{ij} \in L^2 (\mathbb T^d)$,  such
that the sum of these terms will be equal to
$$
\sum\limits_{i,j = 1}^d  \Theta^{ij} \ \frac{\partial^2 u_0 (x)}{\partial x^i \partial x^j}
$$
%\comment{Indices $i,j$ are naughty, they jump up and down}
with a
%positive definite
constant matrix $\Theta = \{\Theta^{ij} \}$. Let us notice that in this sum  only the symmetric part of the matrix $\Theta$ matters. This leads to the following equation for the functions $\varkappa^{ij}_2(\xi)$ for any $i,j = 1, \ldots, d$:
\begin{equation}\label{kappa_2}
(A \varkappa_2^{ij})(\xi) \ = \ \frac{\Theta^{ij}}{\lambda(\xi)} \ - \ \int\limits_{\mathbb R^d} a(z) \mu (\xi -z) \left(\frac12 z^i z^j - z^i \varkappa^{j}_1 (\xi -z) \right) \ dz,
\end{equation}
where $A$ is the same operator as in (\ref{Akappa}). Thus the matrix $\Theta$ is determined from the following solvability condition for equation (\ref{kappa_2}):
\begin{equation}\label{FA2}
\begin{array}{rl}
\displaystyle
\Theta^{i j} \int\limits_{\mathbb T^d} \frac{\mu(\xi)}{\lambda(\xi)} \ d\xi = \tilde \Theta^{i j} \!\!\!\!\!&\displaystyle=  \int\limits_{\mathbb T^d}
\int\limits_{\mathbb R^d}  \frac12 (\xi - q )^i (\xi - q )^j a (\xi-q) \mu (q ) \mu(\xi) \ dq \ d\xi
\\[7mm]
&\displaystyle
-\, \int\limits_{\mathbb T^d}   \int\limits_{\mathbb R^d}   a (\xi-q) \mu (q ) \mu(\xi) (\xi - q )^i \varkappa_1^j (q) \ dq \ d\xi
 \end{array}
\end{equation}
for any $i, j$.

\begin{lemma}\label{l_2}
The integrals on the right-hand side of (\ref{FA2}) converge. Moreover the symmetric part of the matrix
$\Theta = \{ \Theta^{i j} \}$ defined in (\ref{FA2}) is positive definite.
\end{lemma}

\begin{proof}
The first statement of the lemma immediately follows from the existence of the second moment of the function $a(z)$.
Since the integral $ \int\limits_{\mathbb T^d} \frac{\mu(\xi)}{\lambda(\xi)} \ d\xi$ equals  a positive constant, it is sufficient to prove that the symmetric part  of the right-hand side of (\ref{FA2}) is positive definite. To this end we consider the following integrals, symmetric for all $i,j$:
\begin{equation}\label{I}
I^{ij} \ = \ \int\limits_{\mathbb T^d}   \int\limits_{\mathbb R^d}  a (\xi-q) \mu (q ) \mu(\xi) \Big(  (\xi - q )^i + ( \varkappa_1 (\xi) - \varkappa_1 (q))^i \Big)  \Big(  (\xi - q )^j + ( \varkappa_1 (\xi) - \varkappa_1 (q))^j \Big) \ dq \ d\xi,
\end{equation}
and prove that the symmetric part of the right-hand side of (\ref{FA2}) is equal to $I$:
\begin{equation}\label{Ibis}
I^{i j} \ = \ \tilde\Theta^{i j} + \tilde\Theta^{j i} \ =  \int\limits_{\mathbb T^d}   \int\limits_{\mathbb R^d}  (\xi - q )^i (\xi - q )^j a (\xi-q) \mu (q ) \mu(\xi) \ dq \ d\xi \ -
\end{equation}
$$
\int\limits_{\mathbb T^d}   \int\limits_{\mathbb R^d}   a (\xi-q) \mu (q ) \mu(\xi) (\xi - q )^i \varkappa_1^j (q) \ dq \ d\xi -
\int\limits_{\mathbb T^d}   \int\limits_{\mathbb R^d}   a (\xi-q) \mu (q ) \mu(\xi) (\xi - q )^j \varkappa_1^i (q) \ dq \ d\xi.
$$
Using (\ref{L1_2}) we have
$$
\int\limits_{\mathbb T^d}   \int\limits_{\mathbb R^d} (\xi - q )^i a (\xi-q) \mu (q ) \mu(\xi) \varkappa^j_1 (\xi) \ dq \ d\xi = - \int\limits_{\mathbb T^d}   \int\limits_{\mathbb R^d} (\xi - q)^i a (\xi-q) \mu (q ) \mu(\xi) \varkappa^j_1 (q) \ dq \ d\xi.
$$
Consequently,
\begin{equation}\label{last}
\int\limits_{\mathbb T^d}   \int\limits_{\mathbb R^d}  a (\xi-q) \mu (q ) \mu(\xi)  (\xi - q)^i ( \varkappa_1 (\xi) - \varkappa_1 (q))^j \ dq \ d\xi \ = \ - 2 \int\limits_{\mathbb T^d}   \int\limits_{\mathbb R^d} (\xi - q)^i a (\xi-q) \mu (q) \mu(\xi) \varkappa^j_1 (q) \ dq \ d\xi.
\end{equation}
Further, combining equation (\ref{kappa_1bis}) on $\varkappa_1$  with (\ref{L1_1})-(\ref{L1_2}), we get
$$
\int\limits_{\mathbb T^d}   \int\limits_{\mathbb R^d}  a (\xi-q) \mu (q) \mu(\xi)  ( \varkappa_1 (\xi) - \varkappa_1 (q))^i \varkappa^j_1 (\xi)  \ dq \ d\xi  =
$$
$$
\int\limits_{\mathbb T^d}  \mu(\xi)  \varkappa^j_1 (\xi)  \  \int\limits_{\mathbb R^d}  a (\xi-q) \mu (q )  ( \varkappa_1 (\xi) - \varkappa_1 (q))^i \ dq \ d\xi =
$$
$$
- \int\limits_{\mathbb T^d} \int\limits_{\mathbb R^d}  a (\xi-q)  (\xi-q)^i \mu (q) \mu(\xi)  \varkappa^j_1 (\xi)  \ dq \ d\xi  =  \int\limits_{\mathbb T^d}   \int\limits_{\mathbb R^d}  a (\xi-q) (\xi- q)^i \mu (q ) \mu(\xi)  \varkappa^j_1 (q)  \ dq \ d\xi,
$$
and
$$
- \int\limits_{\mathbb T^d}   \int\limits_{\mathbb R^d}  a (\xi-q) \mu (\xi ) \mu(q)  ( \varkappa_1 (\xi) - \varkappa_1 (q))^i \varkappa^j_1 (q)  \ dq \ d\xi \,=
$$
$$
 -
\int\limits_{\mathbb T^d}   \int\limits_{\mathbb R^d}  a (\xi-q) \mu (q ) \mu(\xi)  ( \varkappa_1 (q) - \varkappa_1 (\xi))^i \varkappa^j_1 (\xi)  \ dq \ d\xi  =
$$
$$
\int\limits_{\mathbb T^d}   \int\limits_{\mathbb R^d}  a (\xi-q) \mu (q ) \mu(\xi)  ( \varkappa_1 (\xi) - \varkappa_1 (q))^i \varkappa^j_1 (\xi)  \ dq \ d\xi =  \int\limits_{\mathbb T^d}   \int\limits_{\mathbb R^d}  a (\xi-q) (\xi- q)^i \mu (q ) \mu(\xi)  \varkappa^j_1 (q)  \ dq \ d\xi.
$$
Thus
$$
\int\limits_{\mathbb T^d}   \int\limits_{\mathbb R^d}  a (\xi-q) \mu (\xi) \mu(q) ( \varkappa_1 (\xi) - \varkappa_1 (q))^i  ( \varkappa_1 (\xi) - \varkappa_1 (q))^j \ dq \ d\xi \ =
$$
$$
\int\limits_{\mathbb T^d}   \int\limits_{\mathbb R^d}  a (\xi-q) \mu (\xi ) \mu(q) \Big( (\xi - q)^i \varkappa^j_1 (q) + (\xi - q)^j \varkappa^i_1 (q) \Big) \ dq \ d\xi,
$$
which together with (\ref{I}) and (\ref{last}) implies (\ref{Ibis}).

The structure of (\ref{I}) implies that $(Iv,v) \ge 0, \; \forall \ v \in \mathbb R^d$, and moreover $(Iv,v)>0$ since $\varkappa_1(q)$ is the periodic function while $q$ is the linear function, consequently
$\big[ \big( (\xi - q ) + ( \varkappa_1 (\xi) - \varkappa_1 (q))\big)\cdot v \big]^2 $ can not be identically 0 if $v\not=0$.
\end{proof}

Thus equality (\ref{mle}) follows from (\ref{ml_1}) - (\ref{K2_1}), (\ref{K1_2}) and (\ref{kappa_2}).

The main lemma is proved.
\end{proof}

\section{Proof of Theorem \ref{T1}}
\label{s_prooft1}

\begin{lemma}\label{l_3}
If $\varkappa_1$ and $\varkappa_2$ are the solutions of (\ref{kappa_1}) and (\ref{kappa_2}) respectively, then for all $f \in {\cal S}(\mathbb R^d)$
$$
\| v^{\varepsilon} - u^{\varepsilon}\|_{L^2(\mathbb R^d)} \to 0
$$
as $\varepsilon \to 0$.
\end{lemma}

\begin{proof}
We have from (\ref{mle}) that
$$
L^{\varepsilon} v^{\varepsilon} = \Theta \ \nabla \nabla u_0 + \phi_\varepsilon,
$$
where $\|\phi_\varepsilon \|_{L^2(\mathbb R^d)}  \to 0$ as $\varepsilon \to 0$. Then
$$
(L^{\varepsilon} - m) v^{\varepsilon} + m (v^{\varepsilon} - u_0) = \Theta \ \nabla \nabla u_0 - m u_0 + \phi_\varepsilon = f + \phi_\varepsilon.
$$
Since $\|v^{\varepsilon} - u_0 \|_{L^2(\mathbb R^d)}  \to 0$ due to representation (\ref{v_eps}), we get
\begin{equation}\label{v1}
(L^{\varepsilon} - m) v^{\varepsilon} = f + \tilde \phi_\varepsilon \quad \mbox{ with } \; \| \tilde \phi_\varepsilon \|_{L^2(\mathbb R^d)}  \to 0.
\end{equation}
For the operator $(L^{\varepsilon} - m)^{-1}$ we have
$$
\| (L^{\varepsilon} - m)^{-1} \|_{\mathcal{L}(L^2(\mathbb R^d),L^2(\mathbb R^d))} \ \le \ C
$$
with $C$ being independent of $\varepsilon$. Then using (\ref{v1}) we obtain
$$
u^{\varepsilon} = (L^{\varepsilon} - m)^{-1} f = (L^{\varepsilon} - m)^{-1} \Big( (L^{\varepsilon} - m) v^{\varepsilon} - \tilde \phi_\varepsilon \Big) = v^{\varepsilon} - (L^{\varepsilon} - m)^{-1} \tilde \phi_\varepsilon,
$$
and
$$
\| v^{\varepsilon} - u^{\varepsilon}\|_{L^2(\mathbb R^d)}  = \| (L^{\varepsilon} - m)^{-1} \tilde \phi_\varepsilon \|_{L^2(\mathbb R^d)}  \to 0.
$$
\end{proof}

\begin{corollary}\label{C1}
$$
\| u^{\varepsilon} - u_0 \|_{L^2(\mathbb R^d)} \to 0 \quad \mbox{ as } \; \varepsilon \to 0,
$$
i. e. (\ref{convergence1}) holds for any $f \in \cal{S}$.
\end{corollary}

\begin{proof}[Proof of Theorem 1]
For any $f \in L^2(\mathbb R^d)$ there exists $f_\delta \in \cal{S}$ such that $\| f - f_\delta\|_{L^2(\mathbb R^d)} <\delta$.
Since the operator $(L^\varepsilon - m)^{-1}$ is bounded uniformly in $\varepsilon$, then
\begin{equation}\label{delta_1}
\| u^{\varepsilon}_\delta - u^\varepsilon \|_{L^2(\mathbb R^d)} \le C_1 \delta,
\end{equation}
and
\begin{equation}\label{delta_2}
\| u_{0,\delta} - u_0 \|_{L^2(\mathbb R^d)} \le C_1 \delta,
\end{equation}
where
$$
u^{\varepsilon} \ = \ (L^{\varepsilon} - m)^{-1} f, \; \; u_{0} \ = \ (\hat L - m)^{-1} f, \; \;
u^{\varepsilon}_\delta \ = \ (L^{\varepsilon} - m)^{-1} f_\delta, \; \; u_{0,\delta} \ = \ (\hat L - m)^{-1} f_\delta.
$$
Since $\| u^{\varepsilon}_\delta - u_{0, \delta} \|_{L^2(\mathbb R^d)} \to 0 $ by Corollary \ref{C1}, then (\ref{delta_1}) - (\ref{delta_2}) imply that
$$
\mathop{ \overline{\rm lim}}\limits_{\varepsilon \to 0} \| u^{\varepsilon} - u_0 \|_{L^2(\mathbb R^d)}  \le 2 C_1 \delta
$$
with an arbitrary small $\delta>0$. This implies that $\| u^{\varepsilon} - u_0 \|_{L^2(\mathbb R^d)} \to 0 $, as $\varepsilon \to 0$.
This completes the proof.

\end{proof}

\section{Convergence of semigroups in $L^2(\mathbb R^d)$}
\label{s_semigroup}

%\subsection{The space $L^2(\mathbb R^d)$}

Since for any $\varepsilon>0$ the bounded operator defined in (\ref{L_eps}) is symmetric and negative in $L^2(\mathbb R^d, \nu_\eps)$, where $\nu (y) = \frac{\mu(y)}{\lambda(y)}$ and $\nu_\eps(x) = \nu(\frac{x}{\varepsilon})$, then by the Hille-Yosida theorem it is the generator of a strongly continuous contraction semigroup $T^\varepsilon(t)$ in $L^2(\mathbb R^d, \nu_\eps)$. Denote $T^0(t)$ a strongly continuous contraction semigroup in  $L^2(\mathbb R^d)$ generated by $L^0$.

\begin{proposition}\label{l_4}
For each  $f \in L^2(\mathbb R^d)$ there holds $T^\varepsilon(t)f \to T^0(t) f, \quad \ t \ge 0$. Moreover, this convergence is
uniform on bounded time intervals.

\end{proposition}

\begin{proof} The space $ S(\mathbb R^d)$ is a core for the operator $L^0$.
By the approximation theorem \cite[Ch.1, Theorem 6.1]{EK} it is sufficient to show that for any $u \in S(\mathbb R^d)$  there exists $v^\eps \in L^2(\mathbb R^d, \nu_\eps)$ such that
\begin{equation}\label{conv1}
\| v^{\varepsilon} - u \|_{L^2(\mathbb R^d, \nu_\eps)} \to 0
\end{equation}
and
\begin{equation}\label{conv2}
\| L^{\varepsilon} v^{\varepsilon} - L^0 u \|_{L^2(\mathbb R^d, \nu_\eps)} \to 0
\end{equation}
as $\varepsilon \to 0$.

Notice that under our assumption (\ref{lm})
$$
0< \gamma_1 \le \nu_\eps(x) \le \gamma_2 < \infty.
$$
Therefore,
$$
\gamma_1 \| f \|^2_{L^2(\mathbb R^d)}  \le \| f \|^2_{L^2(\mathbb R^d, \nu_\eps)} \le \gamma_2 \| f \|^2_{L^2(\mathbb R^d)}.
$$
Thus the convergence (\ref{conv1})-(\ref{conv2}) is equivalent to
\begin{equation}\label{conv3}
\| v^{\varepsilon} - u \|_{L^2(\mathbb R^d)} \to 0, \quad
\| L^{\varepsilon} v^{\varepsilon} - L^0 u \|_{L^2(\mathbb R^d)} \to 0.
\end{equation}

For $v^\eps$ we take
$$
v^{\varepsilon}(x) \ = \ u(x)+ \varepsilon \varkappa_1 (\frac{x}{\varepsilon})\cdot \nabla u(x) + \varepsilon^2 \varkappa_2 (\frac{x}{\varepsilon})\cdot \nabla \nabla u(x),
$$
where $\varkappa_1$ and $\varkappa_2$ are the same as in (\ref{v_eps}).
Then the first convergence in (\ref{conv3}) follows by the same arguments as those in the proof of Lemma \ref{fi_0}, and the second one is a consequence of (\ref{mle}). Now the desired statements follow from Theorem 6.1 (Chapter 1) in \cite{EK}.
%if we take for $\pi_\eps$ the identical transformation.

\end{proof}

\begin{corollary}
The convergence of semigroups implies the convergence of solutions of the corresponding evolution equations.
\end{corollary}

\section{Markov semigroup in $C_0 (\mathbb R^d)$}
\label{s_markov}

We consider the Markov semigroup $T(t)$ generated by the operator $L$ given by (\ref{L_u}) in $C_0 (\mathbb R^d)$,
where $C_0 (\mathbb R^d)$ stands for the Banach space of continuous functions vanishing at infinity with the norm $\|f \| = \sup |f(x)|$.
Here we impose on the functions $a(x)$, $\lambda(x)$ and $\mu(x)$ slightly more
restrictive conditions than those of Section \ref{s_pbmset}. Namely, we suppose that %Namely, we suppose that
%We consider a Markov semigroup $T(t)$  generated by the operator (\ref{L_u}), where we assume that
%\begin{itemize}
%\item
\begin{equation}\label{MS1}
a(x) \in C (\mathbb R^d) , \; a(x) = a(-x), \; a(x) \ge 0, \quad a(x) \le \frac{C}{1+|x|^{d+\delta}}, \quad\hbox{with } \delta >2,
\end{equation}
and the functions $\lambda(x), \mu(x) \in C(\mathbb R^d)$ are continuous, periodic and satisfy the same bounds as above:
$$
0< \alpha_1 \le \lambda(x), \  \mu(x) \le \alpha_2 < \infty.
$$
Then the operator $L$ is bounded in $C_0 (\mathbb R^d)$, and $T(t): \ C_0 (\mathbb R^d) \to C_0 (\mathbb R^d)$.

\begin{lemma}\label{C_1}
The semigroup $T(t)$ generated by the operator (\ref{L_u}) is the Feller semigroup, i.e. it is a strongly continuous, positivity preserving, contraction and conservative semigroup in $C_0(\mathbb R^d)$.\\
For each probability measure $\nu$ in $\mathbb R^d$
there exists a jump Markov process $X$ corresponding to the semigroup $T(t)$ with the initial distribution $\nu$ and with a c\`{a}dl\`{a}g modification, i.e. with sample paths in $D_{\mathbb R^d}[0, \infty)$ (right-continuous functions with finite left-hand limits).
\end{lemma}

\begin{proof}
Since $L$ is bounded and  satisfies the positive maximum principle, the statement of Lemma follows from the Hille-Yosida theorem. In addition we can rewrite $L$ as follows:
\begin{equation}\label{MS2}
(L f)(x) \ = \ \tilde\lambda(x) \int\limits_{\mathbb R^d} (f(y) - f(x)) p(x,y) \ dy, \quad \int_{\mathbb R^d} p(x,y) \ dy = 1 \; \forall x
\end{equation}
with
$$
\tilde \lambda(x) = \lambda(x) q(x), \quad q(x) = \int_{\mathbb R^d} a(x-y) \mu(y) dy >0, \quad p(x,y) = \frac{a(x-y) \mu(y)}{q(x)}.
$$
This representation implies that $L$ is a generator of jump Markov process with $T(t)1=1$.

\noindent
The proof of the second statement follows from general results concerning Feller semigroups, see e.g. \cite{BSW, EK}.
\end{proof}

%\begin{lemma}\label{C_2}
%For each probability measure $\nu$ in $\mathbb R^d$
%there exists a jump Markov process $X$ corresponding to the semigroup $T(t)$ with the initial distribution %$\nu$ and with a c\`{a}dl\`{a}g modification, i.e. with sample paths in $D_{\mathbb R^d}[0, \infty)$ %(right-continuous functions with finite left-hand limits).
%\end{lemma}

Let us consider the family of semigroups $T^\varepsilon(t)$ generated by the operators $L^\varepsilon$ defined in (\ref{L_eps}) and the famity of corresponding Markov processes $X_\varepsilon$. We denote by $T^0(t)$ the semigroup in $C_0 (\mathbb R^d)$ generated by the operator $L^0$ given by (\ref{L_hat}). First we prove the result about convergence of the semigroups.

\begin{proposition}\label{C_3}
For each  $f \in C_0 (\mathbb R^d)$ there holds
\begin{equation}\label{MST}
 \lim\limits_{\varepsilon \to 0}  T^\varepsilon(t)f = T^0 (t) f, \quad  t \ge 0.
\end{equation}
Moreover, this convergence is uniform on bounded time intervals.
\end{proposition}

\begin{proof}
We follow the same reasoning as in the proof of Proposition \ref{l_4}, and show the convergence (\ref{conv1})-(\ref{conv2}) in the norm of the Banach space $ C_0 (\mathbb R^d)$. We again take $S(\mathbb R^d)$ as a core for $L^0$ in  $ C_0 (\mathbb R^d)$, and for any $u \in  S (\mathbb R^d)$ consider the approximation sequence $v^\varepsilon $ given by (\ref{v_eps}). First we have to  prove that $v^\varepsilon  \in  C_0 (\mathbb R^d)$. To this end it suffices to show that $\varkappa_1, \ \varkappa_2 \in  C (\mathbb T^d)$, where $\varkappa_1, \ \varkappa_2$ are solutions of equations (\ref{kappa_1}), (\ref{kappa_2}), respectively.  We remind that the equation on $\varkappa_1$ reads
\begin{equation}\label{MSkappa_1}
\int\limits_{\mathbb T^d}  \hat a (\xi-\eta) \mu (\eta)  (\varkappa_1 (\eta) - \varkappa_1 (\xi) ) \ d\eta \ = \ f(\xi),
\end{equation}
where the function $\hat a \in C(\mathbb T^d)$ was defined by (\ref{hata}), and
\begin{equation}\label{MSf}
f(\xi) \ = \ \int\limits_{\mathbb T^d}  \hat b (\xi-\eta) \mu (\eta )  \ d\eta \ \in \ (C(\mathbb T^d))^d, \quad \hat b (\eta) = \sum_{k \in Z^d} a(\eta+k) \ k.
\end{equation}
As was shown above, see (\ref{FA1}), $\int\limits_{\mathbb T^d} f(\xi) \mu(\xi) \ d \xi \  = \ 0$, consequently
the solvability condition holds, and there exists a solution  $\varkappa_1 \in  (L^2 (\mathbb T^d))^d$ of equation (\ref{MSkappa_1}). We will show now that  $\varkappa_1^{i} \in  C (\mathbb T^d)$ for all $i = 1, \ldots, d$.

Let us rewrite equation (\ref{MSkappa_1}) for each $i$ as follows
\begin{equation}\label{MSP}
(P-E)\varkappa_1^{i} \ = \ g^{i}, \quad g^{i}(\xi) = \frac{f^{i}(\xi)}{q(\xi)}, \quad q (\xi) =  \int\limits_{\mathbb T^d} \hat a (\xi-\eta) \mu (\eta ) \ d\eta >0
\end{equation}
with
\begin{equation}\label{Akappa_torus}
(P \varphi) (\xi) \ = \ \int\limits_{\mathbb T^d}  p (\xi, \eta) \varphi (\eta) \ d\eta, \quad p(\xi, \eta) = \frac{\hat a(\xi - \eta) \mu(\eta)}{q(\xi)}, \quad  \int\limits_{\mathbb T^d}  p (\xi, \eta) \ d\eta = 1 \;\; \forall \; \xi.
\end{equation}
Then $P$ is a compact operator with positivity improving property in $C(\mathbb T^d)$, and by the Krein-Rutman theorem there exists the maximal eigenvalue $\lambda_0=1$ corresponding to the eigenfunction $\varphi_0(\eta) \equiv 1$, and other eigenvalues of $P$ are less than 1 by  the absolute value. Consequently, $C(\mathbb T^d)  = \{ 1 \} \oplus {\cal H}_1$ with
$$
{\cal H}_1 =\Big\{ \psi \in C(\mathbb T^d): \  \int\limits_{\mathbb T^d} \mu (\eta ) q(\eta) \psi(\eta) \ d\eta =0 \Big\}.
$$
One can easily check that ${\cal H}_1$ is an invariant subspace for $P$.
Using Neumann decomposition for the operator $P_1 - E$ with $P_1 = P|_{{\cal H}_1}$ we can see that the operator $P-E$ is an invertable operator on ${\cal H}_1$ mapping ${\cal H}_1$ on itself. Thus,
$$
\varkappa_1 \  = \ -(E- P_1)^{-1} g \ \in C(\mathbb T^d).
$$

Similarly, we get $\varkappa_2^{ij} \in C(\mathbb T^d), \; \forall \ i,j =1, \ldots, d,$ for the solutions of equation (\ref{kappa_2}).
Thus  $v^\varepsilon  \in  C_0 (\mathbb R^d)$, and (\ref{v_eps}) implies that
$$
\| v^\varepsilon - u \|_{C_0 (\mathbb R^d)} \ \to \ 0 \quad \mbox{ as } \; \varepsilon \to 0.
$$
The convergence
$$
\| L^\varepsilon v^\varepsilon - L^0 u \|_{C_0 (\mathbb R^d)} \ = \ \| \phi_\varepsilon \|_{C_0 (\mathbb R^d)} \to \ 0 \quad \mbox{ as } \; \varepsilon \to 0
$$
follows from the same reasoning as in the proof of Lemma \ref{ml}. Thus we can apply the approximation theorem from \cite{EK} in the same way as in Proposition \ref{l_4}, and obtain convergence (\ref{MST}).
\end{proof}

Applying the same arguments as in the proof of the last statement one can show that the homogenization result
of Theorem \ref{T1} also holds in the space of continuous functions.

\begin{proposition}\label{also_in_c}
Under the assumptions of this section, for any $f\in C_0(\mathbb R^d)$  we have
$$
 \| (L^{\varepsilon} - m)^{-1} f - (L^0 - m)^{-1} f \|_{C_0(\mathbb R^d)} \ \to 0, \quad \mbox{  as } \; \varepsilon \to 0.
$$
\end{proposition}

We proceed with the main result of this section that states the invariance principle for the family of processes $X_\eps$.

\begin{theorem}\label{C_4} (Invariance principle).
Let $X_\varepsilon$ be a Markov process corresponding to the semigroup $T^\varepsilon(t)$ with an initial distribution $\nu$, and $X_0$ be a Markov process corresponding to the semigroup $T^0 (t)$  with the same initial distribution. Then the Markov processes $X_\varepsilon$ and $X_0$ have sample paths in $D_{\mathbb R^d}[0, \infty)$, and $X_\varepsilon \ \Rightarrow \ X_0$ in $D_{\mathbb R^d}[0, \infty)$.
\end{theorem}

\begin{proof}
The fact that $X_\eps$ has a modification in $D_{\mathbb R^d}[0, \infty)$ has been justified in Lemma \ref{C_1}. The limit process $X_0$ is a diffusion process that has continuous trajectories.
The convergence in distributions $X_\varepsilon \ \Rightarrow \ X_0$ in the paths space $D_{\mathbb R^d}[0, \infty)$ follows from Proposition \ref{C_3} that gives the convergence of finite-dimensional distributions,
and Theorem 2.5 (Chapter 4) from \cite{EK}.
%The statement of Lemma \ref{C_4} immediately follows from Lemmas \ref{C_1} - \ref{C_3} and Theorem 2.5
%from \cite{EK}, see below, concerning general results on convergence of the Feller semigroups.

%\begin{theorem}\label{C_5} (Chapter 4, Theorem 2.5, \cite{EK})
%Let $E$ be locally compact and separable. For $n=1,2, \ldots$ let $\{T_n(t) \}$ be a Feller semigroups on
%$C_0 (E)$, and suppose $X_n$ is a Markov process corresponding to $\{ T_n(t) \}$ with sample paths in
%$D_{E} [0, \infty)$. Suppose that $\{ T(t) \}$ is a Feller semigroup on $C_0(E)$ and that for each
%$f \in C_0(E)$
%$$
%\lim\limits_{n \to \infty}  T_n (t) f \  = \  T (t) f, \quad  t \ge 0.
%$$
%If $\{ X_n (0) \}$ has limiting distribution $\nu$, then there is a Markov process $X$ corresponding to
%$\{ T(t) \}$ with initial distribution $\nu$ and sample paths in $D_{E} [0, \infty)$, and $X_n \ \Rightarrow %\ X$ in  $D_{E} [0, \infty)$.
%\end{theorem}

\end{proof}


\begin{thebibliography}{10}

\bibitem{BLP} Bensoussan A., Lions J.L., Papanicolaou G. \emph{Asymptotic Analysis for Periodic Structures}. North Holland, Amsterdam, 1978.

\bibitem{BSW} B\"{o}ttcher B., Schilling R., Wang J., L\'{e}vy Matters III: L\'{e}vy-Type Processes: Construction, Approximation and Sample Path Properties, Springer, 2009.

\bibitem{BCF} C. Brandle, E. Chasseigne, R. Ferreira, Unbounded solutions of the nonlocal heat equation, Commun. Pure Appl. Anal. 2011, 10, pp. 1663-166


\bibitem{CCR} E. Chasseigne, M. Chaves, J. Rossi, Asymptotic behavior for nonlocal diffusion equations, J. Math. Pures Appl. 86 (2006), 271-291


\bibitem{EK} S. N. Ethier, T. G. Kurtz, Markov processes: Characterization and convergence. Wiley $\&$ Sons, 2005.


\bibitem{FKKMZ}  D. Finkelshtein, Yu. Kondratiev, O. Kutoviy, S. Molchanov, E. Zhizhina,
Density behavior of spatial birth-and-death stochastic evolution of mutating genotypes
under selection rates, Russian Journal of Math. Physics, 2014, vol. 21, No 4, pp. 450-459.


\bibitem{Franke2013}
B. Franke, A functional non-central limit theorem for jump-diffusions with periodic coefficients driven
by stable Levy-noise, {\it Journal of Theoretical Probability,}  {\bf 20} (2007), 1087--1100.



%\bibitem{CCR} E. Chasseigne, M. Chaves, J. Rossi, Asymptotic behavior for nonlocal diffusion equations, J. %Math. Pures Appl. 86 (2006), 271-291


\bibitem{JKO}
Jikov, V.V., Kozlov O.A., Oleinik, O.A. \emph{Homogenization of Differential Operators and Integral Functionals}. Springer, New York, 1994.

\bibitem{KKP} Yu. Kondratiev, O. Kutoviy, S. Pirogov, Correlation functions
and invariant measures in continuous contact model, Ininite
Dimensional Analysis, Quantum Probability and Related Topics Vol.
11, No. 2 (2008) 231-258.

\bibitem{KPZ} Yu. Kondratiev, S. Pirogov, E. Zhizhina, A Quasispecies Continuous Contact Model
in a Critical Regime, Journal of Statistical Physics, 163(2), 357-373 (2016)


\bibitem{Rho_Var2008} R. Rhodes and V. Vargas, Scaling limits for symmetric Itˆo-L´evy processes in random
medium. {\it Stochastic Process. Appl.}, {\bf 119}(12) (2009),  4004--4033.


\bibitem{Sandric2016}  N. Sandri\'c, Homogenization of periodic diffusion with small jumps,
   arXiv:1510.06140v1

%\bibitem{PP}

\end{thebibliography}
\end{document}